\documentclass[12pt,reqno]{amsart}
\usepackage{amssymb,amsmath,color}
\usepackage{soul}
\usepackage{color}
\def\({\left(}
\def\){\right)}

\def\R {\mathbb{R}}

\newcommand{\be}{\begin{equation} }
\newcommand{\ee}{\end{equation} }

\def \and{\qquad\text{and}\qquad}

\def\Bbb{\mathbb}

\def\R{\Bbb R}

\def \no#1#2#3 {{\bf #1} (#3), #2.}
\def \eds#1#2#3 {#1, #2, #3.}
\newtheorem{proposition}{Proposition}[section]
\newtheorem{theorem}[proposition]{Theorem}

\newtheorem{lemma}[proposition]{Lemma}

\newtheorem{remark}[proposition]{Remark}

\numberwithin{equation}{section}

\title[Blow Up of Solutions to Semilinear Wave Equations ]
{Blow Up of Solutions to Semilinear Non-autonomous Wave Equations Under Robin Boundary Conditions }
\author[]{J. Kalantarova}

\address{ Department of Mathematics, Izmir University of Economics,
Sakarya Caddesi, No:156,
Izmir, Turkey}
\email{jamila.kalantarova@ieu.edu.tr}
\begin{document}
%%%%%%%%%%%%%%%%%%%%%%%%%%%%%%%%%%%%
\begin{abstract}
The problem of blow up of solutions to the initial boundary value
problem for non-autonomous semilinear wave equation with damping and accelerating terms  under the Robin boundary
condition is studied. Sufficient conditions of blow up in a finite time of solutions to semilinear damped wave equations with arbitrary large initial energy are obtained. A result on blow up of solutions with negative initial energy of semilinear second order wave equation with accelerating term is also obtained.
\end{abstract}

%\end{frontmatter}

 \maketitle

\section{Introduction} In this paper we present some results about the global non existence of
solutions of the initial boundary value problem for second order nonlinear wave equations under Robin boundary conditions:
\begin{equation}\label{a1}
 u_{tt} + b u_t =\Delta u +f(u)+h(x,t),\quad  x\in\Omega,\quad    t>0,
\end{equation}
\begin{equation}\label{a2}
\frac{\partial u}{\partial\nu}+\gamma u=0,\qquad
x\in\partial\Omega,\quad t>0,
\end{equation}
\begin{equation} \label{a3}
u(x,0)=u_0(x),\qquad u_t(x,0)=u_1(x), \quad x\in\Omega.
\end{equation}
Here $\Omega\subset \R^n$ is a bounded domain with smooth boundary $\partial \Omega$, $\nu$ is a unit otward vector to the bundary $\partial \Omega$ ,$b\in\mathbb{R} $ and $\gamma \in\R$ are
gien numbers. The given source term $h$, the function $f(\cdot)$ and the initial data $u_0,u_1$ are so smooth that the problem \eqref{a1}-\eqref{a3} has a classical local (in time) solution. We also assume that
\begin{equation}\label{hc}
h_0:=\int_0^\infty\|h(t)\|^2dt<\infty, \quad h_1:=\sup_{t\in \R^+}\|h(t)\| <\infty,
\end{equation}
 and  the
nonlinear term $f\in C^1(\R^1\rightarrow \R^1)$ satisfies the condition
\begin{equation}\label{nnl1}
f(s)s-2(2\alpha+1)F(s)\geq 0, \ \ \forall s\in \R,
\end{equation}
with some $\alpha >0$. Here and below $F(s)=\int_0^s f(\tau)d \tau,$  \ $\|\cdot\|$ and $(\cdot,\cdot)$ denote the norm and inner product in $L^2(\Omega)$ respectively.
For existence of a local classical solution of initial boundary value problems for nonlinear wave equations under boundary conditions covering the Robin boundary conditions, see, e.g., \cite{Weid} and references therein.\\
The purpose of this paper is to show that for some class of initial data the local classical solutions to the problem \eqref{a1}-\eqref{a3} blow up in a finite time.
%and source term

There have been many works devoted to the problems of global non-existence and  blow up of solutions to initial boundary value problems for nonlinear wave equations (see, e.g., \cite{Gla},\cite{Lev1}, \cite{KaLa}, \cite{Kor}, \cite{MiPo},\cite{Str} and references therein).
In majority of these papers sufficient conditions of blow up of solutions in a finite time of initial boundary value problems for various nonlinear wave equations under the homogeneous Dirichlet or Neumann boundary conditions, nonlinear boundary conditions and dynamic boundary conditions  are provided.
 A number of papers were addressed  to the question of blow up of solutions with arbitrary positive initial energy of initial boundary value problems for various  nonlinear wave equations (see, e.g.\cite{Kor} ,\cite{Kut}, \cite{Wu}
 and references therein).

The novelty of results we obtained, compared to preceding results on blow up of solutions of nonlinear wave equations, is that  we obtained results on blow up of solutions for more wide class of non-autonomous equations under the Robin boundary conditions. For weakly damped nonlinear wave equation (when $b>0,$ $\gamma>0$) we show that there are solutions with arbitrary large initial energy that blow up in a finite time. We also obtained sufficient conditions of blow up of solutions of the semilinear wave equation with accelerating term (as far as we know it is a first result of this type for nonlinear wave equation obtained employing energy method).

The main tool we used in the proof of our results is the concavity method and its modification.

In what follows we will employ the following Lemma.
\begin{lemma}\label{Levine} (see \cite{Lev1})
 Let $\Psi (t)$ be a positive,
twice differentiable function, which satisfies, for $t>0,$ the inequality
\begin{equation}\label{levin}
\Psi ^{''}(t)\Psi (t)-(1+\alpha )\left[ \Psi\rq{}(t)\right] ^2\geq 0
\end{equation}
with some $\alpha >0.$ If $\Psi (0)>0$ and $\Psi'(0)>0,$ then there
exists a time $t_0\leq \frac{\Psi (0)}{\alpha \Psi'(0)}$ such that
$\Psi (t)\rightarrow +\infty $ as $t\rightarrow t_0.$
\end{lemma}
\noindent and its modification:
\begin{lemma}\label{lemKaLa} ( \cite{KaLa}) Let twice continuously differentiable
function $\Psi(t)$ satisfy the inequality
\be\label{01} \Psi''(t) \Psi(t)-(1+\alpha)\left[\Psi(t)\right]^2\geq-
2 C_1\Psi(t)\Psi'(t)-C_2\Psi^2(t), \ \ t\ge 0 \ee and
\be\label{01} \Psi(0)>0,
\Psi'(0)>-\gamma_2\alpha^{-1}\Psi(0), \ee where $\alpha >0,
C_1,C_2\geq 0, C_1+C_2>0$ and $\gamma_2=-C_1-\sqrt{C_1^2+\alpha
C_2}$. Then there exists $$t_1 \leq T_1=\left(2\sqrt{C_1^2+\alpha
C_2}\right)^{-1}\ln\frac{\gamma_1\Psi(0)+\alpha\Psi'(0)}{\gamma_2\Psi(0)+\alpha\Psi'(0)},
$$
with $\gamma_1=-C_1+\sqrt{C_1^2+\alpha C_2}$ such that
$$
\Psi(t)\rightarrow \infty \ \ \mbox{as} \ \  t\rightarrow t_1^{-}.
$$
\end{lemma}
%\begin{lemma}\label{Korpusov} (see \cite{Kor})
%Suppose $\Psi(t)\in C^2[0,T]$, satisfies inequality
%\begin{equation}\label{2.1}
%\Psi\Psi''-(1+\alpha)(\Psi')^2+C_1\Psi'\Psi\geq -C_2\Psi,
%\quad\alpha>0,\;C_1\geq  0,\;C-2\geq  0,
%\end{equation} and
%\begin{gather}\label{2.18}
%\Psi'(0)>\frac{\kappa}{\alpha}\Psi(0), \  \big(\Psi'(0)-\frac{\kappa}{\alpha}\Psi(0)\big)^2
%>\frac{2\beta}{2\alpha+1}\Psi(0),
%\end{gather}
%where $\Psi(t)\geq  0$, $\Psi(0)>0$. Then there exists
%$T_{0}<\infty$ such that  $$\limsup_{t\rightarrow T_0^-}\Psi(t)=+\infty.$$
%\end{lemma}
\noindent We will also use the Poincar\'e inequality
\begin{equation}\label{Pnk}
\int_{\Omega} u^2 dx\leq d_0\left(\int_{\partial\Omega}u^2d\sigma+\int_{\Omega}|\nabla u|^2 dx\right)
\end{equation}
and the following
inequality
\begin{equation}\label{Pnk1}
\int_{\partial \Omega}u^2d\sigma \leq \epsilon\|\nabla
u\|^2+C(\epsilon)\|u\|^2,
\end{equation}
where $\Omega\subset \R^n$ is a bounded domain with the boundary $\partial \Omega$, $\epsilon>0$ can be
chosen small enough, and $C(\epsilon)>0$ depends on $\epsilon$ (see,
e.g.\cite{Lad}, page 34).

\section{Damped semilinear wave equation under the Robin boundary condition}
In this section we will find sufficient conditions of global nonexistence of
solutions to the problem \eqref{a1}-\eqref{a3} with $b>0, \gamma >0$ under some
restrictions on initial functions.

The main result obtained in this section is the following theorem.
\begin{theorem}\label{T1}
Suppose that $u$ is a local solution of the problem \eqref{a1}-\eqref{a3} and one of the following conditions is satisfied
\be\label{con1}
(u_0,u_1)>\frac{b}{\alpha}\|u_0\|^2 \ \mbox{ if} \ K_0 =4(1+2\alpha)E(0)+A_0 \le0, \ee
or
\be\label{con2}
2(u_0,u_1)>\alpha^{-1}\left(\frac b2+\sqrt{\frac{b^2}4+\alpha}\right)(\|u_0\|^2+K_0),\ \mbox{if} \ K_0>0,
\ee
%or \be\label{con3} b< 0, \gamma>0, \ \mbox{ and} \ \beta=4(1+2\alpha)E(0)+A_0<0,\ee
%, where $\alpha\geq \frac{1}{2}$ is a given number and $d_0$ is defined in \eqref{cond1}.
where $$E(0)=\frac12\|u_1\|^2+\frac12\|\nabla u_0\|^2 +\frac{\gamma}{2}\int_{\partial\Omega}u_0^2d\sigma -(F(u_0),1),$$  \be\label{A0}A_0=
\frac {1+2\alpha}{b}h_0+\frac{d_0h^2_1}{4\alpha \min\{1,\gamma\}}.
\ee
Then there exists $t_1<\infty$ such that
$$\lim_{t\rightarrow t_1 ^{-}}\|u(t)\|=\infty.$$
\end{theorem}
\begin{proof}
Taking scalar product of \eqref{a1}  (in $L^2(\Omega)$) with $u_t$ we obtain the
energy equality:
\begin{equation}\label{Aa1}
\frac{d}{dt}E(t)+b \|u_t(t)\|^2=(u_t(t),h(t)),
\end{equation}
where
\begin{equation}\label{En2}
E(t)=\frac{1}{2}\|u_t(t)\|^2+\frac{1}{2}\|\nabla
u(t)\|^2+\frac{\gamma}{2}\int_{\partial\Omega}u^2(x,t)
d\sigma-(F(u(t)),1).
\end{equation}
%\begin{equation}\label{Ee1}
%\frac{d}{dt}E(t)\leq\frac{1}{4b}\|h\|^2.
%\end{equation}
Integrating \eqref{Aa1} over the interval $(0,t)$ we get:
\begin{equation}\label{En1}
E(t)=E(0)-b\int_0^t \|u_\tau (\tau)\|^2 d\tau+\int_0^t (u_\tau (\tau),h(\tau))d\tau.
\end{equation}

Set $$\Psi(t)= \|u(t)\|^2+c_0,$$ where $u$ is a solution of the problem \eqref{a1}-\eqref{a3} and $c_0$ is a non-negative parameter which will be determined. Employing equation \eqref{a1} and the boundary condition \eqref{a2} we obtain
\begin{equation}\label{Aa2}
\Psi''(t)=2\|u_t\|^2-2\|\nabla u\|^2-2\gamma \int_{\partial\Omega}u^2d\sigma-2b(u,u_t)+2(f(u),u)+2(u,h).
\end{equation}
Since $2(u(t),u_t(t))=\Psi'(t)$, by using the energy equality  \eqref{En1}  we have:
\begin{multline*}
\begin{split}
\Psi''(t)&+b\Psi'(t)\geq  2\|u_t\|^2-2\|\nabla u\|^2-2\gamma \int_{\partial\Omega}u^2d\sigma+4(1+2\alpha)(F(u),1)
\\&+2(u,h)=4(1+2\alpha)\left[-\frac12\|u_t\|^2-\frac12\|\nabla u\|^2-\frac\gamma2\int_{\partial\Omega}u^2d\sigma+(F(u),1)\right]
\\&+2(h,u)+4(1+\alpha)\|u_t\|^2+4\alpha\|\nabla u\|^2+4\alpha\gamma\int_{\partial\Omega}u^2d\sigma.
\end{split}
\end{multline*}
Substituting  the value of $E(t)$ from \eqref{En1} into  the right hand side of  the last inequality we get
\begin{multline}\label{psi2a}
\begin{split}
\Psi''(t)&+b\Psi'(t)\geq-4(1+2\alpha)E(0)+4(1+2\alpha)b\int_0^t \|u_\tau (\tau)\|^2 d\tau++2(h,u)\\&-4(1+2\alpha)\int_0^t(u_\tau,h)d\tau
+4(1+\alpha)\|u_t\|^2+4\alpha\|\nabla u\|^2+4\alpha\gamma\int_{\partial\Omega}u^2d\sigma.
\end{split}
\end{multline}
%\begin{equation}
%\Psi''(t)+b\Psi'(t)\geq -4(1+2\alpha)E(t)+2(h,u).
%\end{equation}
Thanks to the Young's inequality and  the Poincar\'e inequality   \eqref{Pnk} we have
\be\label{psi2b}
\Big|\int_0^t(u_\tau(\tau),h(\tau))d\tau\Big|\le b\int_0^t\|u_\tau(\tau)\|^2d\tau+\frac{1}{4b}
\int_0^t\|h(\tau)\|^2d\tau
\ee
and
\be\label{psi2c}
2\|h\|\|u\|\leq\epsilon\|u\|^2+\frac{1}{\epsilon}\|h\|^2\leq\epsilon d_0\left(\|\nabla u\|^2+\int_{\partial\Omega}u^2d\sigma\right)+\frac{1}{\epsilon}\|h\|^2.
\ee
Employing \eqref{psi2b} in \eqref{psi2a} we get
\begin{multline*}\label{eqn6}
\Psi''(t)+b\Psi'(t)\geq -4(1+2\alpha)E(0)-\frac{1+2\alpha}{b}\int_{0}^{t}\|h(\tau)\|^2d\tau+4(1+\alpha)\|u_t\|^2
\\+4\alpha\|\nabla u(t)\|^2
 +4\alpha\gamma\int_{\partial\Omega}u^2d\sigma-2\|h(t)\|\|u(t)\|.
\end{multline*}
By using  on the right hand side of the last inequality the inequality \eqref{psi2c}  with \\ $\epsilon =4\alpha \min\{1,\gamma\}d_0^{-1}$ we obtain
\begin{equation}\label{in1}
\Psi''(t)+b\Psi'(t)\geq -4(1+2\alpha)E(0)-A_0+4(\alpha+1) \|u_t\|^2,
\end{equation}
where $A_0$ is defined in \eqref{A0}.

First consider the case when the initial data satisfy the condition \eqref{con1}. In this case we choose $c_0=0$ and obtain from \eqref{in1} the inequality
$$
\Psi''(t)\Psi(t)-(\alpha+1)\left[\Psi'(t)\right]^2\ge -b\Psi'(t)\Psi(t)+4(\alpha+1) \|u_t\|^2\Psi(t)-(\alpha+1)\left[\Psi'(t)\right]^2.
$$
It remains to note  that due to Schwarz inequality $$4\|u_t\|^2\Psi(t)\geq \left[\Psi'(t)\right]^2,$$ and therefore
$$
\Psi''(t)\Psi(t)-(\alpha+1)\left[\Psi'(t)\right]^2\ge -b\Psi'(t)\Psi(t).
$$

Then Lemma \ref{lemKaLa} guaranties that $\|u(t)\|$ tends to infinity in a finite time.\\
If  the condition \eqref{con2} is satisfied, i.e. $K_0>0.$  we choose $c_0=K_0$ and deduce from \eqref{in1} the inequality
\be\label{last}
\Psi''(t)\Psi(t)-4(\alpha+1)\|u_t\|^2\Psi(t)\geq -b\Psi(t)\Psi'(t)-\left[\Psi(t)\right]^2.
\ee
 Thus the  inequality \eqref{last} implies that $\Psi(t)$ satisfies the inequality \eqref{01} with $C_1=\frac b2$ and $C_2=1$. The conclusion of the Theorem follows in this case from Lemma \ref{lemKaLa}.
\end{proof}
\begin{remark}
Notice that if  the nonlinear term has the form $f(u)=|u|^p u,$ $p>0,$ then we can find infinitely many  initial data with arbitrary positive initial energy for which the corresponding solutions blow up in a finite time. In this case
$F(u)=\frac{1}{p+2}|u|^{p+2}$ and the condition \eqref{nnl1} is satisfied with $\alpha=\frac{p}{4}$.
For sufficiently smooth nonzero $u_0$ and
 $$u_1(x)=\left(\frac{2}{p+2}\|u_0\|_{L^{p+2}}^{p+2}\right)^{1/2} \frac{u_0(x)}{\|u_0\|}$$
the initial energy takes the form:
%$u_1=\lambda u_0$
$$
E(0)=\frac{1}{2}\|\nabla u_0\|^2+\frac{\gamma}{2}\int_{\partial\Omega}u_0 ^2 d\sigma$$
and  the condition (ii)   of Theorem \ref{T1} takes the form

\begin{multline*}
\begin{split}
 (ii) \ 2(u_1,&u_0)=2\left(\frac{2}{p+2}\|u_0\|_{L^{p+2}}^{p+2}\right)^{1/2}\|u_0\|\\&\geq\frac1\alpha\left(\frac{b}{2}+\sqrt{\frac {b^2}4+\alpha}\right)\left[\|u_0\|^2
+2(1+2\alpha)\|\nabla u_0\|^2+2\gamma(1+2\alpha)\int_{\partial \Omega}u_0^2d\sigma+A_0\right].
\end{split}
\end{multline*}
Since $p>0$ we can choose appropriate $u_0$ for which the initial energy is arbitrary positive number and the conditions (i), (ii) are satisfied.
Thus corresponding solutions will exist only on a finite interval.
\end{remark}

% $$\left[2r\|u_0\|^2 -\frac{b}{\alpha}\|u_0\|^2\right]^2>\frac{2d_0}{2\alpha+1}\|u_0\|^2$$
% $$\left(2 r-\frac{b}{\alpha}\right)^2 \|u_0\|^2>
% \left[2(2\alpha+1)r^2
 %\frac{r^2}{2}\|u_0\|^2+\frac{1}{2}\|\nabla u_0\|^2+\frac{\gamma}{2}\int_{\partial\Omega}u_0 ^2 d\sigma-\frac{1}{p+2}\int_{\Omega}|u_0|^{p+2} (x)dx
 %\right]
 %\frac{2(b+2)}{2\alpha+1}
% $$
% $$\left(4r^2-\frac{br}{\alpha} +\frac{b^2}{\alpha^2}\right)>>2(2\alpha+1)r^2, \alpha\leq\frac{1}{2}
% $$
% \left[
 %\frac{1}{2}\|\nabla u_0\|^2+\frac{\gamma}{2}\int_{\partial\Omega}u_0 ^2 d\sigma-\frac{1}{p+2}\int_{\Omega}|u_0|^{p+2} (x)dx
 %\right]\frac{2(b+2)}{2\alpha+1}$$
\section{Blow up of solutions of semilinear non-autonomous wave equations with accelerating term}
   Now we consider the initial boundary value problem for a semilinear wave equation with accelerating term, i.e. the problem \eqref{a1}-\eqref{a3}  when $\gamma \in \mathbb{R}$ and $b<0$. Let us note that when at least one of the numbers $\gamma$ or
$b$ is negative we can not  directly use the concavity method to get sufficient condition for blow up of solutions to the problem \eqref{a1}-\eqref{a3}. Therefore we make the following change of variables:
\begin{equation}\label{change}
u(x,t)=e^{mt}v(x,t),
\end{equation}
where $m$ is some positive parameter to be determined.
%$$
%(mb+m^2)e^{mt}v +(b+2m)e^{mt}v_t +e^{mt}v_{tt}=e^{mt}\Delta v+f(e^{mt}v)+h(x,t).
%$$
Then we obtain the following problem for the function $v(x,t).$\\
%Thus the function $v(x,t)$ defined by \eqref{change} is a solution of the problem:
\begin{equation}\label{z1}
(mb+m^2)v +(b+2m)v_t +v_{tt}=\Delta v
+e^{-mt}f(e^{mt}v)+e^{-mt}h(x,t),
\end{equation}
\begin{equation}\label{z2}
\frac{\partial v}{\partial\nu}+\gamma v=0,\qquad
x\in\partial\Omega,\quad t>0,
\end{equation}
\begin{equation}\label{z3}
v(x,0)=u_0(x),\quad v_t (x,0)= u_1(x) - m u_0(x).
\end{equation}
The main result of this section is the following theorem.
\begin{theorem}\label{BlowUp2}
 Suppose that the condition
\eqref{nnl1} holds, and
\be\label{condV}
4(\alpha+1)E_1(0)-\frac{h_0}{2m\alpha}-
\frac1{2(mb+m^2)\alpha}h_1^2-4(\alpha +1)(b+2m)\|u_0\|^2\geq
0,
\ee
where
$$
E_1(0)=-\frac{mb+m^2}{2} \|u_0\|^2-\frac{1}{2}\|u_1 - mu_0\|^2 -\frac{1}{2}\|\nabla u_0\|^2-\frac{\gamma}{2}\int_{\partial\Omega} u_0 ^2d\sigma+ (F(u_0),1),
$$
 $m$ is a positive
solution of the equation
\begin{equation}\label{mc}
m^2+mb-|\gamma|C(|\gamma|^{-1})=0,
\end{equation}
and $C(|\gamma|^{-1})$ is the constant in the inequality \eqref{Pnk}. Then the corresponding solution of the problem \eqref{a1}-\eqref{a3}
can exist only on a finite interval $[0,T)$.
\end{theorem}
\begin{proof}
Taking scalar product of  \eqref{z1} with $v_t$ and by using the equality
$$
\frac{\partial}{\partial t}\left[e^{-2mt}F(e^{mt}v)\right]+2me^{-2mt}F(e^{mt}v)-me^{-mt}f(e^{mt}v)v=e^{-mt}f(e^{mt}v)v_t
$$
we obtain
%\begin{multline*}
%(mb+m^2)\int_{\Omega}v v_t dx +(b+2m)\int_{\Omega}{v_t}^2 dx+\int_{\Omega}v_{tt} v_tdx\\=\int_{\Omega}\Delta v v_t dx +\int_{\Omega}e^{-mt}f(e^{mt}v)v_t dx+\int_{\Omega}e^{-mt}h(x,t)v_t dx. \end{multline*}
%\begin{multline*}
\begin{multline*}
\begin{split}
\frac{mb+m^2}2\frac{d}{dt}\|v\|^2+&(b+2m)\|v_t\|^2 dx
+\frac{1}{2}\frac{d}{dt}\|v_{t}\|^2\\&=-\frac{1}{2}\frac{d}{dt}\|\nabla v\|^2
-\frac{\gamma}{2}\frac{d}{dt}\int_{\partial\Omega}v^2 d\sigma
+\frac{d}{dt}[e^{-2mt}(F(e^{mt}v),1)] \\&+ 2me^{-2mt}(F(e^{mt}v),1)
-me^{-mt}(f(e^{mt}v),v) +e^{-mt}(h,v_t).
\end{split}
\end{multline*}
From the last inequality by using Young's equality
%$$
%\frac{\partial}{\partial t}\left[e^{-2mt}F(e^{mt}v)\right]+2me^{-2mt}F(e^{mt}v)-me^{-mt}f(e^{mt}v)v
%$$
we obtain
\begin{multline}\label{r1}
-\frac{d}{dt}E_1(t)
+(b+2m)\|v_t\|^2 -2m e^{-2mt}(F(e^{mt}v),1)+me^{-mt}(f(e^{-mt}v),v)\\
\leq\varepsilon_1 \|v_t\|^2+\frac{1}{4\varepsilon_1}\|h\|^2 e^{-2mt},
\end{multline}
where
\be\label{E1}
E_1(t):= -\frac{mb +m^2}{2}\|v\|^2
-\frac{1}{2}\|v_t\|^2-\frac{1}{2}\|\nabla v\|^2-\frac{\gamma}{2}
\int_{\partial\Omega}v^2 d\sigma +
e^{-2mt}(F(e^{mt}v),1).
\ee
Thanks to \eqref{nnl1} we have:
$$
e^{-mt}f(e^{mt}v)v = e^{-2mt}f(e^{mt}v)e^{mt}v  \geq 2(2\alpha
+1)e^{-2mt}F(e^{mt}v).
$$
By using this inequality in \eqref{r1}
 we obtain
$$
-\frac{d}{dt}E_1(t)+ (b+2m)\|v_t\|^2
+4\alpha m e^{-2mt}(F(e^{mt}v),1)
\leq \varepsilon_1
\|v_t\|^2+\frac{1}{4\varepsilon_1}\|h\|^2 e^{-2mt}.
$$
We can rewrite the last inequality in the following form
\begin{multline}\label{en1}
%\begin{split}
\frac{d}{dt}E_1(t) \geq 4m\alpha E_1(t)+ 2m\alpha \left[(mb+m^2)\|v\|^2 +\|\nabla v\|^2 +\gamma\int_{\partial\Omega} v^2 d\sigma\right]\\+(-\varepsilon_1+(b+2m)+2m\alpha)\|v_t\|^2
-\frac{1}{4\varepsilon_1}\|h\|^2 e^{-2mt}.
%\end{split}
\end{multline}

Employing the Poincar\'e inequality \eqref{Pnk}
we get from \eqref{en1} the estimate
\begin{multline*}
\begin{split}
-\frac{d}{dt}E_1(t)\geq 4m\alpha E_1(t) &+
(2m\alpha +b+2m -\varepsilon_1)\|v_t\|^2
\\&
+2m\alpha
\left[(mb+m^2)-|\gamma|C(|\gamma|^{-1})\right]\|v\|^2-\frac{1}{4\varepsilon_1}\|h\|^2
e^{-2mt}.
\end{split}
\end{multline*}
Taking in the last inequality $\varepsilon_1=2m\alpha$, and
integrating it we obtain the following estimate from below for
$E_1(t)$.
\begin{multline}\label{ener2}
E_1(t)\geq e^{4m\alpha t}E_1(0)+ (b+2m)e^{4m\alpha t}\int_{0}^{t}
\|v_\tau (\tau)\|^2 e^{-4m\tau}d\tau\\ -\frac{1}{8m\alpha} e^{4m\alpha t}
\int_{0}^{t}\|h(\tau)\|^2 e^{-m(4\alpha+2)\tau}d\tau.
\end{multline}
Let us consider the following function
$$
\Psi(t)= \|v(t)\|^2 + (b+2m)\int_{0}^{t} \|v(\tau)\|^2 d\tau + c_0,
$$
where $v$ is the solution of the problem \eqref{z1}-\eqref{z3} and $c_0$ is a positive
parameter to be chosen later.\\
It is easy to see that
\begin{multline*}
\begin{split}
\Psi'(t)&=2(v(t),v_t(t))+(b+2m)\|v(t)\|^2\\
&=2(v(t),v_t(t))+2(b+2m)\int_{0}^{t}(v(\tau),v_{\tau}(\tau))d\tau
+(b+2m)\|v_0\|^2
\end{split}
\end{multline*}
and
\begin{multline*}
\begin{split}
\Psi''(t)&=2\|v_t(t)\|^2+2 (v(t),v_{tt}(t))+2(b+2m)(v(t),v_{t}(t))\\
&=2\|v_t(t)\|^2+2(v_{tt}(t)+(b+2m)v_t(t), v(t)).
\end{split}
\end{multline*}
Employing here the equation \eqref{z1} and the condition \eqref{nnl1} we obtain
\begin{multline*}
\begin{split}
\Psi''(t)&=2(\Delta v(t)+e^{-mt} f(e^{mt} v(t))+e^{-mt}h -(mb+m^2)v(t), v(t))+2\|v_t(t)\|^2\\&
\geq -2(mb+m^2)\|v(t)\|^2-2\|\nabla v(t)\|^2-2\gamma
\int_{\partial\Omega}v^2d\sigma +
4(2\alpha +1)e^{-2mt}(F(e^{mt}v(t)))
\end{split}\\+2e^{-mt}(h(t),v(t))+2\|v_t(t)\|^2.
\end{multline*}
The last inequality we can rewrite in the form:
\begin{multline}\label{P22}
\Psi''(t)\geq4(2\alpha
+1)E_1(t)+4(mb+m^2)\alpha \|v\|^2\\+4\alpha \|\nabla v\|^2-
4\alpha\gamma\int_{\partial\Omega}v^2 d\sigma+
2e^{-mt}(h,v)
+4(\alpha+1)\|v_t\|^2.
\end{multline}
By using the inequality
$$
2e^{-mt}(h,v)\geq-2(mb+m^2)\alpha
\|v\|^2-e^{-2mt}\frac{1}{2(mb+m^2)\alpha}\|h\|^2
$$
and the notation \eqref{E1} we obtain from \eqref{P22} the estimate
$$
\Psi''(t)\geq 4(2\alpha+1)E_1(t)-e^{-2mt}\frac{1}{2(mb+m^2)\alpha}
\|h\|^2 + 4(\alpha+1)\|v_t\|^2.
$$
From the last inequality due to \eqref{ener2} we have
\begin{multline*}
\Psi''(t)\geq 4(\alpha+1)\left[(b+2m)\int_{0}^{t}\|v_\tau(\tau)\|^2 d\tau+
\|v_t\|^2+c_0\right]\\
+4(\alpha+1)e^{4m\alpha
t}\left[E_1(0)-\frac{1}{2m\alpha}\int_{0}^{t}
e^{-m(4\alpha+2)\tau}\|h(\tau)\|^2 d\tau\right]\\
-4(\alpha+1)c_0-e^{-2m t}\frac{1}{2(mb+m^2)\alpha}\|h(t)\|^2.
\end{multline*}
Thanks to  the condition \eqref{condV} we infer from the last inequality
the following estimate from below for the function  $\Psi''(t)$:
$$
\Psi''(t)\geq 4(\alpha+1)\left[(b+2m)\int_{0}^{t}\|v_\tau(\tau)\|^2 d\tau+
\|v_t(t)\|^2+c_0\right].
$$
Thus employing the Schwarz inequality we get
$$
\Psi''(t)\Psi(t)-(\alpha +1)\left[\Psi'(t)\right]^2\geq 0.
$$
So the statement of the theorem follows from the Lemma \ref{Levine}.
\end{proof}

\end{document}